\documentclass[11pt]{amsart}
\usepackage{mystyle}
\usepackage{comment}
\usepackage{makecell}
\usepackage{tabularx}
\usepackage{tikz}
\usepackage{tikz-cd}
\usepackage[normalem]{ulem}

\newcommand{\A}{\mb{A}}
\newcommand{\B}{\mb{B}}
\newcommand{\I}{\mathrm{I}}
\newcommand{\s}{\ms s}
\newcommand{\U}{\mathbf U}
\renewcommand{\d}{\ms d}
\renewcommand{\H}{\ms H}
\renewcommand{\E}{\ms E}

\title{Moderate-doubling sets in $\mathbb{F}_2^n$ intersect subspaces}

\author{Alex Cohen}
\address{Alex Cohen, Courant Institute, New York University. New York, NY, USA.}
\email{alexcohen@nyu.edu}

\author{Dmitrii Zakharov}
\address{Dmitrii Zakharov, Department of Mathematics, Massachusetts Institute of Technology. Cambridge, MA, USA.}
\email{zakhdm@mit.edu}

\date{\today}

\begin{document}

\begin{abstract}
    We show that any set $A$ in $\F_2^n$ with $|A+A| \le |A|^{2-\eta}$ must intersect a subspace of dimension $O_{\eta}(\log |A|)$ in at least $|A|^{\eta - o(1)}$ elements.
\end{abstract}

\maketitle 

\section{Introduction}
Given a set $A$ in an abelian group $G$, the sumset is defined as 
\[
A + A := \{a + a'\, :\, a, a' \in A\}. 
\]
The size of the sumset lives in the range $[|A|, |A|^2]$. 
If the sumset is small, what can we say about the structure of $A$?
In this paper we focus on $G = \F_2^n$. 
In this context, Freiman's theorem~\cites{Freiman1973, GreenRuzsa2007} says that if $|A+A| \leq K |A|$, then $A$ can be covered by $O_K(1)$ translates of a subspace with size $O_K(|A|)$. 
Marton conjectured that the constants in this theorem can be taken polynomially large in $K$. This is also known as the polynomial Freiman--Ruzsa (PFR) conjecture. 
In a recent breakthrough work, Gowers, Green, Manners and Tao~\cite{GowersGreenMannersTao2025} solved this conjecture over $\F_2^n$ by proving
\[
|A+A| \leq K|A| \Longrightarrow \text{$A$ can be covered by $2K^{12}$ cosets of a subspace with size at most $|A|$.}
\]
It is easy to cover $A$ with $|A|$ many translates of the $\{0\}$ subspace, so this theorem is interesting when $|A+A| < |A|^{1+1/12}$.\footnote{Liao \cite{Liao2024} reduced the constant from 12 to 9 using a modification of their argument. On the other hand, it is known that the PFR exponent must be at least 1.4.} 

By contrast, not much is known in the \textit{moderate doubling} regime where $|A+A| = |A|^{2-\varepsilon}$.\footnote{In \cites{BatemanKatz2011, Shkredov2013, Shkredov2014} some structural results in the moderate doubling regime are obtained under certain assumptions on higher energies.}
Over $\F_2^n$, we prove moderate doubling sets have large intersection with a bounded size subspace. 
\begin{theorem}\label{thm:combinatorial-corollary}
For all $\varepsilon >0$ there exists $L>0$ such that the following holds. 
For any set $A \subset \F_2^n$ with $|A+A| = |A|^{2-\eta}$, there exists a subspace $V \subset \F_2^n$ such that $|V| \le |A|^{L}$ and 
\[
\E_{a \in A} \log |A \cap (V+a)| \ge (\eta-\varepsilon)\log |A|.
\]
\end{theorem}
In particular, some translate of $V$ intersects $A$ in at least $|A|^{\eta - \varepsilon}$ points. A related result was proven by Bateman and Katz \cite[Section 5]{BatemanKatz2012}: they showed that if $|A+A| \le |A|^{2-\eta}$ then $|A\cap V| \ge |A|^{\eta-o(1)} \dim V$ for some subspace $V$ of dimension at most $|A|^{\frac{1-\eta}{2}-o(1)}$. Our Theorem \ref{thm:combinatorial-corollary} improves the bound on the dimension from polynomial to logarithmic.

We consider three examples of sets with moderate doubling.

\begin{itemize}
    \item \textbf{Random subset of a subspace}. Let $V \subset \F_2^n$ be a subspace of size $N^{ \frac{2}{1+\varepsilon} }$ and let $A \subset V$ be a uniformly random subset of size $N$. Then $|A+A| \le |A|^{2-\varepsilon}$.

    \item \textbf{Union of cosets}. Let $V \subset \F_2^n$ be a subspace of size $N^\varepsilon$, let $\Lambda \subset \F_2^n$ be a random set of size $N^{1-\varepsilon}$, and let $A = V+\Lambda$. Then $|A+A| \le |A|^{2-\varepsilon}$.

    \item \textbf{Bernoulli}. Let $p \in (0,1/2]$ and, for sufficiently large $n$, let $A \subset \F_2^n$ be the set of vectors with at most $pn$ non-zero coordinates. 
    The size of $A$ is estimated by 
    \[
    \log |A| = (1+o_n(1)) n H(p)
    \]
    where $H(p) = -p \log p - (1-p)\log (1-p) $ is the binary entropy function.
    Vectors in $A+A$ have at most $2pn$ non-zero coordinates, so 
    \[
    \log |A+A| = (1+o_n(1)) n H(2p).
    \]
    For small $p$, we may approximate $H(p) = p \log (e/p) + o_p(p)$, leading to the doubling estimate 
    \[
    \log \frac{|A|^2}{|A+A|} = (1+o_n(1))(2p \log (e/p) - 2p\log(e/2p)+o_p(p)) =  2p (\log 2+o(1)).
    \]
    It follows that $|A+A| \le |A|^{2-\varepsilon}$ with $\varepsilon = \frac{2+o(1)}{\log_2 (1/p)}$.
\end{itemize}

The first two examples play an important role in Gowers--Green--Manners--Tao's proof of the PFR conjecture over $\F_2$.
% of the PFR conjecture over $\F_2$. 
The Bernoulli example represents a new behavior in the moderate doubling regime.
One can combine all three examples: take a bunch of cosets of some subspace $V$, put the Bernoulli example inside each coset, and then take a random subset of the result. 
These are pretty much the only examples we know of sets with moderate doubling in $\F_2^n$. 
All these examples have polynomially large intersections with subspaces, agreeing with Theorem~\ref{thm:combinatorial-corollary}.

Quantitatively, the Bernoulli example shows that the smallest possible $L$ in Theorem \ref{thm:combinatorial-corollary} must be at least exponential in $\varepsilon^{-1}$. On the other hand, our proof provides an upper bound $L \le \varepsilon^{-C\varepsilon^{-1}}$ (see Section \ref{sec:inductive-step}).

%See Remark \blue{TODO} for a discussion of quantitative aspects of this result. 
Theorem~\ref{thm:combinatorial-corollary} follows from a more precise entropic statement, which we discuss next. The entropic approach to additive combinatorics was introduced by Ruzsa \cite{Ruzsa2009} and developed by Tao~\cite{Tao2010} and has proven to be a very useful tool in the field \cites{GreenMannersTao2025, GreenRuzsa2019, Kontoyiannis2014} and in its applications \cite{Hochman2014}. 
% The resolution of the PFR conjecture over $\F_2$ by Green, Gowers, Manners, Tao crucially relies on a \textit{fibring inequality} for entropic doubling, which does not have an adequate combinatorial analogue. For the same reason, we work with entropic doubling. 
% We introduce the fibring inequality and other tools from entropic sumset calculus in Section \ref{subsec:pfr-tools}.

Let $X$ be a random variable taking values in an abelian group $G$. 
The Shannon entropy of $X$, defined in~\eqref{eq:ShannonEntropy}, measures the size of $X$. 
For example, if $A$ is a subset of $G$ and $\U_A$ is a random variable uniformly distributed over $A$, then $\H[\U_A] = \log |A|$.
The {\em entropic doubling constant} of $X$ is defined by $\sigma =\H[X_1+X_2]-\H[X]$, where $X_1$ and $X_2$ are independent copies of $X$.  In the case $X = \U_A$,
\[
\H[X_1+X_2] \le  \log |A+A| \le \log |A| + \log K = \H[X] +\log K,
\]
so the entropic doubling constant $\sigma$ is upper bounded by the logarithm of the usual doubling constant $K$. If we know that $|A+A| \le |A|^{2-\varepsilon}$, this implies that $\H[X_1+X_2] \le (2-\varepsilon) \H[X]$. 
The following theorem is our main result about entropic doubling.
We denote the quotient map by a subspace $V$ as $\pi_V: G \to G/V$.  
% For a subspace $V \subset G$ we let $\pi_V: G\to G/V$ denote the quotient map.

\begin{theorem}\label{thm:main-entropy}
    For every $\varepsilon>0$ there exists $L>0$ with the following property. 
    Let $X$ and $Y$ be independent random variables on $G=\F_2^n$. Then there exists a subspace $V \subset G$ with $\H[\U_V] \le L (\H[X]+\H[Y])$ and such that
    \[
    \H[ \pi_V(X) + \pi_V(Y) ] \ge \H[\pi_V(X)] + \H[\pi_V(Y)] - \varepsilon (\H[X] + \H[Y]).
    \]
\end{theorem}

In words, after quotioning out by a subspace of dimension $O(\H[X]+\H[Y])$ we can kill off almost all of the additive interaction between variables $X$ and $Y$. For example, one can take both $X, Y$ to be independent copies of $\U_A$ for some subset $A \subset G$. 
Below are some corollaries of this theorem.  

\begin{corollary}\label{cor:rich-cosets}
    For every $\varepsilon >0$ there exists $L>0$ with the following property. 
    Let $X$ and $Y$ be independent random variables on $G=\F_2^n$ and suppose that $\H[X+Y] = \H[X]+\H[Y] - s$ for some $s\ge 0$. Then there exists a subspace $V \subset G$ with $\H[\U_V] \le L (\H[X]+\H[Y])$ and such that
    \begin{equation}\label{eq:cor-inequality}
    \H[ X | \pi_V(X) ], \H[Y|\pi_V(Y)] \ge s - \varepsilon(\H[X]+\H[Y]).    
    \end{equation}
\end{corollary}

\begin{proof}[Proof of Theorem \ref{thm:combinatorial-corollary}]
     Let $A$ be a set with $|A+A| \le |A|^{2-\eta}$, and set $X$ and $Y$ to be independent copies of $\U_A$. Then we have $\H[X+Y] \le (2-\eta) \log |A| = \H[X]+\H[Y] -\eta \log |A|$.
     So for any $\varepsilon>0$, there exists a subspace $V$  such that $\H[\U_V] \le L \H[\U_A]$, i.e. $|V| \le |A|^L$, and that $\H[ \U_A~|~ \pi_V(\U_A) ] \ge (\eta-\varepsilon) \log |A|$. On the other hand, one can verify $\H[ \U_A~|~ \pi_V(\U_A) ] = \E_{a \in A} \log |A \cap (V+a)|$ giving the desired estimate.
\end{proof}

By writing $\H[\pi_V(X)] = \H[X+\U_V] - \H[\U_V]$, (\ref{eq:cor-inequality}) can be restated as
\begin{equation}\label{eq:conjecture3.1}
\H[X+\U_V] - \H[\U_V] \le \H[X+Y] - \H[Y] + \varepsilon(\H[X]+\H[Y]).    
\end{equation}
It was conjectured in \cite[Conjecture 3.1]{Abbe2024} that for any $\varepsilon >0$ and any identically distributed $X$ and $Y$ there exists a subspace $V$ of dimension at most $e^{C\varepsilon^{-1}}\H[X]$ such that (\ref{eq:conjecture3.1}) holds. Corollary \ref{cor:rich-cosets} implies this with a slightly worse estimate $\varepsilon^{-C\varepsilon^{-1}}$. 

Similar results hold for sums of more than two variables. 

\begin{corollary}
    \label{cor:many-sums}
    For every $k \ge 2, \varepsilon>0$ there exists $L>0$ with the following property. 
    Let $X_1, \ldots, X_k$ be independent random variables on $G=\F_2^n$. Then there exists a subspace $V \subset G$ with $\H[\U_V] \le L (\H[X_1]+\ldots+\H[X_k])$ and such that
    \[
    \H[ \pi_V(X_1) +\ldots + \pi_V(X_k) ] \ge \H[\pi_V(X_1)] +\ldots + \H[\pi_V(X_k)] - \varepsilon (\H[X_1] +\ldots + \H[X_k]).
    \]
\end{corollary}

Taking $X_1, \ldots, X_k$ all copies of the same random variable $X$, we get the following dichotomy: for any $\varepsilon>0$ there is a subspace $V$ of dimension $O_{\varepsilon,k}(\H[X])$, such that one of the following holds:
\begin{itemize}
    \item we have $\H[\pi_V(X)] \le \varepsilon \H[X]$,
    \item or the variable $\tilde X=\pi_V(X)$ is $(k, \varepsilon)$-dissociated, i.e. $\H[\tilde X_1+\ldots +\tilde X_k] \ge (k-\varepsilon) \H[\tilde X]$.
\end{itemize}

Our proof of Theorem \ref{thm:main-entropy} builds on the inductive strategy of Gowers, Green, Manners and Tao \cite{GowersGreenMannersTao2025}.
Let $X,Y$ be a pair of independent variables and suppose that $\H[X+Y] = (1-\eta)(\H[X]+\H[Y])$. 
Our goal is to find subspace structure as long as $\eta > 0$. 
The base case is $\eta = 1/2$, where it is known that $X$ and $Y$ must be uniformly distributed over a subspace. 
If $\eta$ is just slightly smaller than $1/2$, the polynomial Freiman--Ruzsa conjecture implies that $X$ and $Y$ are close to a subspace $V$. 
For $\eta < 1/2$, we employ induction on $\eta$. 
Let $X_1,X_2,Y_1,Y_2$ be independent copies of $X$ and $Y$. Following~\cite{GowersGreenMannersTao2025}, we consider four moves that replace $(X, Y)$ with a new pair of independent random variables: 
\begin{itemize}
    \item First sumset $X_1+Y_2, X_2+Y_1$,
    \item First fiber $(X_1 ~|~  X_1+Y_2 = u), (Y_1 ~|~ Y_1+X_2=w)$,
    \item Second sumset $X_1+X_2, Y_1+Y_2$,
    \item Second fiber $(X_1 ~|~  X_1+X_2 = u), (Y_1 ~|~ Y_1+Y_2=w)$.
\end{itemize}
If any of these four moves improves the doubling ratio of $X, Y$, then the inductive hypothesis lets us find subspace structure between the new variables. 
This reduces the problem to the case when none of the four moves improve the doubling ratio. 
In~\cite{GowersGreenMannersTao2025}, the authors use a \textit{fibring identity}, Eq.~\eqref{eq:FibringIdentity}, and a special identity valid over $\F_2$, to show that the sums of fibers $X_u =(X_1 | X_1+Y_2=u)$ and $Y_w=(Y_1 | Y_1+X_2=w)$ have a much smaller doubling ratio and so their induction can be closed. 
This last step is called the ``Endgame'' in their paper.

The key new difficulty in implementing this strategy for moderate doubling sets is the fiber move. Suppose we manage to prove Theorem \ref{thm:main-entropy} for each pair of variables $X_u = (X_1 ~|~  X_1+Y_2 = u), Y_w= (Y_1 ~|~ Y_1+X_2=w)$, where $u \sim X_1+Y_2$ and $w \sim Y_1+X_2$ are arbitrary. This means that we obtain a subspace $V(u, w)$ which explains the additive interaction between $X_u$ and $Y_w$. To complete the inductive step we then need to show that there is a single subspace $V$ explaining the interaction between $X$ and $Y$.
In Lemma \ref{lem:YSizeLowerBd} we show that this is the case. The idea is to fix a typical $u$ from $X_1+Y_2$, sample an independent sequence $w^{(1)}, \ldots, w^{(k)}$ from $Y_1+X_2$, and set 
\[
V = V(u, w^{(1)}) + \dots + V(u, w^{(k)}).
\]
For a carefully chosen value of $k$, we show that this subspace satisfies $\H[\pi_V(Y)] \le (1-c)\H[Y]$, completing the inductive step. 

In \Cref{sec:preliminaries} we recall basic definitions from entropic calculus and tools such as the fibring identity and Balog--Szemer\'edi--Gowers theorem. In \Cref{sec:endgame} we prove a version of the endgame from \cite{GowersGreenMannersTao2025}. In \Cref{sec:setup-induction} we set up intermediate statements $\A$ and $\B$ which will be used in the inductive proof. In \Cref{sec:inductive-step} we execute the inductive step and deduce the main result and its corollaries. 

\subsection{Acknowledgements.} We thank Ben Green and Terence Tao for valuable comments and for bringing the lemma of Bateman and Katz \cite{BatemanKatz2012} and Conjecture 3.1 from \cite{Abbe2024} to our attention. 

\section{Preliminaries}\label{sec:preliminaries}

\subsection{Definitions and basic identities}\label{subsec:definitions}
Throughout this paper we work in the group $G = \F_2^n$. 
We work with $G$-valued random variables, typically denoted $X, Y$. 
Our convention is that $X_j, Y_j$ denote independent copies of $X$ and $Y$. 

The entropy of a $G$-valued random variable is defined by 
\begin{equation}\label{eq:ShannonEntropy}
\H[X] := \sum_{x \in G} \Pr[X = x] (- \log \Pr[X = x]), 
\end{equation}
where if $\Pr[X = x] = 0$, the summand is set to zero. 

For a set $V$, we denote by $\U_V$ a random variable uniformly distributed over $V$. 

For possibly dependent random variables $X$ and $U$, we define
\begin{equation}\label{eq:ChainRuleCondEntropy}
\H[X | U] := \E_{u \sim U} \H[X | U = u] = \H[X, U] - \H[U]. 
\end{equation}
The equality between the left and right sides is the \textit{chain rule for entropy}. 
Whenever we condition on random variables, we implicitly take an expectation over their values.

For two possibly dependent random variables $X$ and $Y$, the mutual information is defined by 
\begin{align}
\I[X : Y] &:= \H[X] + \H[Y] - \H[X, Y] \nonumber \\ 
&= \H[X] - \H[X | Y] = \H[Y] - \H[Y | X]. \label{eq:MutualInformationConditionalEntropy}
\end{align}
If $S$ is a third random variable, the conditional mutual information is defined by
\begin{align}
\I[X : Y | S] &:= \E_{s \sim S} \I[X : Y | S = s] = \H[X | S] + \H[Y | S] - \H[X, Y | S] \nonumber\\ 
&= \H[X, S] + \H[Y, S] - \H[X, Y, S] - \H[S] \geq 0.\label{eq:Submodularity}
\end{align}
The mutual information is always nonnegative, and is zero if and only if $X$ and $Y$ are independent. The nonnegativity of the second line is known as \textit{submodularity}. 

We denote the quotient map by a subspace $V$ as 
\[
\pi_V: G \to G / V.
\]
We may also refer to this map by $\pi_{G/V}$.
We have
\[
\H[X + \U_V | \pi_V(X+\U_V)] = \H[\U_V], 
\]
so 
\begin{equation}\label{eq:EntropyOfProjection}
\H[\pi_V(X)] = \H[\pi_V(X+\U_V)] = \H[X + \U_V] - \H[\U_V].
\end{equation}
Let $V_1, V_2$ be two subspaces. Applying submodularity, Eq.~\eqref{eq:Submodularity}, to 
\[
(S, X, Y) = (X + \U_{V_1\cap V_2}, \U_{V_1}, \U_{V_2}), 
\]
yields
\[
\H[X + \U_{V_1}] + \H[X + \U_{V_2}] \geq \H[X + \U_{V_1} + \U_{V_2}] - \H[X + \U_{V_1\cap V_2}], 
\]
and subtracting off $\H[\U_{V_1}] + \H[\U_{V_2}] = \H[\U_{V_1+V_2}] + \H[\U_{V_1 \cap V_2}]$ gives 
\begin{equation}\label{eq:SubspaceSubmodularity}
     \H[\pi_{V_1}(X)] + \H[\pi_{V_2}(X)] \geq  \H[\pi_{V_1+V_2}(X)] + \H[\pi_{V_1\cap V_2}(X)] . 
\end{equation}

Let $X$ and $Y$ be possibly dependent random variables, and let $X', Y'$ be mutually independent copies of $X$ and $Y$. The \textit{doubling mass} between $X$ and $Y$ is defined by
\[
\s[X ; Y] := \H[X'] + \H[Y'] - \H[X'+Y']. 
\]
By the chain rule for entropy, Eq. ~\eqref{eq:ChainRuleCondEntropy}, if $X$ and $Y$ are independent then 
\begin{equation}\label{eq:DoublingMassFiberSize}
\s[X ; Y] = \H[X | X+Y] = \H[Y | X+Y].
\end{equation}
There is a trivial upper bound
\begin{equation}\label{eq:TrivialEstimateDoublingMass}
\s[A ; B] \leq \min\{\H[A], \H[B]\}.
\end{equation}

Let $(X, U)$ and $(Y, W)$ be pairs of random variables, the only dependencies being between $(X, U)$ and $(Y, W)$. Let
\[
X_u = (X | U = u),\qquad Y_w = (Y | W = w). 
\]
To be precise, $X_u$ and $Y_w$ are defined to be new random variables, independent of all the previous ones, with distributions equal to the conditional probabilities above. 
The conditional doubling mass is defined by 
\begin{align*}
\s[X \mid U ; Y \mid W] &:= \E_{u\sim U, w\sim W} \s[X_u ; Y_w] \\
&= \H[X | U] + \H[Y | W] - \H[X+Y | U,W].
\end{align*}
Let $V$ be a subspace and let $\pi = \pi_V$. 
Applying the entropic chain rule to $(X, \pi(X))$ gives the identity  
\begin{equation}\label{eq:EntropyFibring}
\H[A] = \H[\pi(A)] + \H[A | \pi(A)]. 
\end{equation}
Applying this identity to $X, Y$, and $X+Y$ gives
\begin{align*}
\H[X] + \H[Y] - \H[X+Y]  &= \H[\pi(X)] + \H[\pi(Y)] - \H[\pi(X+Y)] + \\ 
&\qquad \H[X | \pi(X)] + \H[Y | \pi(Y)] - \H[X+Y | \pi(X)+\pi(Y)].
\end{align*}
By the definition of mutual information, the last term can be written as
\begin{align*}
    \H[X+Y | \pi(X)+\pi(Y)] = \H[X + Y | \pi(X), \pi(Y)] + \I[X+Y : (\pi(X), \pi(Y))\, |\, \pi(X+Y)],
\end{align*}
yielding the fibring identity (cf. Proposition 4.1 in \cite{GowersGreenMannersTao2025}):
\begin{equation}\label{eq:FibringIdentity}
s[X;Y] = s[\pi(X); \pi(Y)] + s[X | \pi(X); Y | \pi(Y)] - \I[X+Y : (\pi(X), \pi(Y)) | \pi(X+Y)].
\end{equation}
Often we will just use the non-negativity of mutual information to bound the left-hand side by the right-hand side.
The term $ s[X | \pi(X); Y | \pi(Y)] $ accounts for the amount of additive interaction between $X$ and $Y$ through the subspace $V$. 

Let $W \subset V$ be two subspaces. Take an expectation over $\pi_V(X) , \pi_V(Y)$ and apply the fibring inequality to find
\begin{equation}
\label{eq:FiberInteraction}
\begin{split}
\s[X | \pi_V(X); Y | \pi_V(Y)] &= \E_{t \sim \pi_V(X) , r \sim \pi_V(Y)} \s[X_t ; Y_r]  \\ 
&\leq \E_{t,r} \bigl(\s[X_t | \pi_W(X_t) ; Y_r | \pi_W(Y_r)]  + \s[\pi_W(X_t) ; \pi_W(Y_r)]\bigr)  \\ 
&=  \s[X | \pi_W(X) ; Y | \pi_W(Y)] + \s[\pi_W(X) | \pi_V(X) ; \pi_W(Y) | \pi_V(Y)] .
\end{split}
\end{equation}
In this inequality, the interaction of $X$ and $Y$ through $V$ is bounded by the interaction through $W$, plus the interaction through $V \bmod W$. 

\subsection{Tools from Gowers--Green--Manners--Tao~\cite{GowersGreenMannersTao2025}}\label{subsec:pfr-tools}

For a pair of variables $X, Y$ we define the Ruzsa distance 
\[
\d[X ; Y] = \H[X' + Y'] - \frac{1}{2} \H[X'] - \frac{1}{2}\H[Y'],
\]
where $X', Y'$ are independent copies of $X$ and $Y$, respectively. 
There are trivial lower and upper bounds 
\[
\frac{1}{2}\H[X] - \frac{1}{2}\H[Y] \leq \d[X ; Y] \leq \frac{1}{2}\H[X] + \frac{1}{2}\H[Y]. 
\]

\begin{theorem}[Entropic Marton's Conjecture \cite{GowersGreenMannersTao2025}*{Theorem 1.8}]\label{thm:PFR}
Let $G = \F_2^n$ and suppose that $X, Y$ are $G$-valued random variables. There exists a subspace $V \subset G$ such that 
\[
\max\{\d[X; \U_V], \d[Y; \U_V]\} \leq 6 \d[X; Y].
\]
\end{theorem}

\begin{corollary}\label{cor:KillingSubspacePFR}
Let $G = \F_2^n$ and suppose that $X, Y$ are $G$-valued random variables. There exists a subspace $V \subset G$ such that $\H[\U_V] \leq 7(\H[X] + \H[Y])$ and 
\[
\max\{\H[\pi_V(X)] , \H[\pi_V(Y)]\} \leq 12 \d[X ; Y]. 
\]
\end{corollary}
\begin{proof}
Let $V$ be the subspace provided by~\Cref{thm:PFR}. We have 
\begin{align*}
\d[X ; \U_V] &= \H[X + \U_V] - \frac{1}{2}\H[X] -\frac{1}{2} \H[\U_V] \\ 
&\stackrel{\eqref{eq:EntropyOfProjection}}{=} \H[\pi_V(X)] + \frac{1}{2}\H[\U_V] - \frac{1}{2}\H[X].
\end{align*}
Using the trivial lower bound on Ruzsa distance, 
\[
\H[\pi_V(X)] = \d[X ; \U_V] + \frac{1}{2}\H[X] -  \frac{1}{2}\H[\U_V] \leq 2 \d[X ; \U_V] \leq 12 \d[X ; Y]. 
\]
The other trivial lower bound, $\d[X ; \U_V] \geq \frac{1}{2}\H[\U_V] - \frac{1}{2}\H[X] $, gives
\[
\H[\U_V] \leq \H[X] + 2\d[X ; \U_V] \leq \H[X] +  12 \d[X ; Y] \leq 7(\H[X] + \H[Y]). 
\]
\end{proof}

\begin{lemma}[Entropic Balog--Szemer\'edi--Gowers \cite{GowersGreenMannersTao2025}*{Lemma A.2}]\label{lem:BSG}
Let $(A,B)$ be a $G^2$-valued random variable (note that $A$ and $B$ are not necessarily independent). Then
\[
\E_{t\sim A+B} \d[A | A+B = t; B | A+B = t] \leq 3 \I[A : B] + 2\H[A+B] - \H[A] - \H[B]. 
\]
\end{lemma}

\section{Endgame estimates}\label{sec:endgame}
Let $X$ and $Y$ be independent $G$-valued random variables. We let $X_j, Y_j$ be independent copies of $X, Y$. 
We apply the fibring identity, Eq.~\eqref{eq:FibringIdentity}, to 
\begin{align}
    A = X_1 \times X_2 \subset G\times G, \quad B = Y_1 \times Y_2 \subset G\times G\label{eq:FirstPairForFibring}
\end{align}
and to 
\begin{align}
    A = X_1 \times Y_2 \subset G\times G,\quad B = Y_1 \times X_2 \subset G\times G.\label{eq:SecondPairForFibring}
\end{align}
In both of these cases, $s[A : B] = 2s[X : Y]$.
Let $S = X_1+X_2+Y_1+Y_2$. 
Applying~\eqref{eq:FibringIdentity} to~\eqref{eq:FirstPairForFibring} gives
\begin{align*}
    \s[X_1+X_2 ; Y_1 + Y_2] + \s[(X_1, X_2) | X_1+X_2; (Y_1, Y_2) | Y_1+Y_2] &= 2\s[X : Y] + \\
    &\qquad \I[(X_1+Y_1,X_2+Y_2) : (X_1+X_2, Y_1 + Y_2) | S].
\end{align*}
We can rewrite the second term on the left-hand side as 
\[
\s[(X_1, X_2) | X_1+X_2; (Y_1, Y_2) | Y_1+Y_2] = s[X_1 | X_1+X_2; Y_1 | Y_1+Y_2]
\]
leading to 
\begin{equation}\label{eq:FirstFibring}
\s[X_1+X_2 ; Y_1 + Y_2] + \s[X_1 | X_1+X_2; Y_1 | Y_1+Y_2] = 2\s[X : Y] + \I[(X_1+Y_1,X_2+Y_2) : (X_1+X_2, Y_1 + Y_2) | S].
\end{equation}
Similarly, applying~\eqref{eq:FibringIdentity} to~\eqref{eq:SecondPairForFibring} gives
\begin{equation}\label{eq:SecondFibring}
\s[X_1+Y_2 ; X_2+Y_1] + \s[X_1 | X_1+Y_2; Y_1 | Y_1+X_2] = 2\s[X ; Y] + \I[(X_1+Y_1, X_2+Y_2) : (X_1+Y_2 : X_2 + Y_1) | S].
\end{equation}
Let 
\begin{equation}\label{eq:Zdefn}
Z_1 := X_1 + Y_1,\quad Z_2 := X_2+Y_1,\quad Z_3 := X_1+X_2,\quad S := X_1+X_2+Y_1+Y_2.
\end{equation}
Because we are working in $\F_2$, $Z_1+Z_2+Z_3 = 0$. The mutual informations above can be expressed in terms of $Z_1,Z_2,Z_3$, and $S$,
\begin{align}
    \I[(X_1 + Y_1, X_2 + Y_2) : (X_1+X_2, Y_1+Y_2) | S] &= \I[Z_1 : Z_3 | S], \label{eq:FirstMutualInf} \\ 
    \I[(X_1+Y_1, X_2+Y_2) : (X_1+Y_2 : X_2 + Y_1) | S] &= \I[Z_1 : Z_2 | S].\label{eq:SecondMutualInf}
\end{align}
The following lemma roughly says that either
\begin{itemize}
    \item One of the four pairs of random variables 
    \[
    (X_1+X_2, Y_1+Y_2),\; (X_1+Y_2, Y_1+X_2),\; (X_1 | X_1 + X_2, Y_1 | Y_1+Y_2),\;\text{or}\; (X_1 | X_1+Y_2, Y_1 | Y_1+X_2)
    \]
    has a larger doubling mass than $(X, Y)$, or 
    \item The fibers $X_u = (X_1 | X_1+Y_2 = u)$ and $Y_w = (Y_1 | Y_1+X_2 = w)$ are typically contained in a small number of cosets. 
\end{itemize}
\begin{lemma}[Endgame Lemma]\label{lem:Endgame}
Let $\eta \in (0, 1/2]$ and $\kappa > 0$. Let $X, Y$ be independent $G$-valued random variables such that 
\[
s[X ; Y] \geq \eta (\H[X] + \H[Y]).
\]
Let $X_1, X_2$ and $Y_1, Y_2$ denote independent copies of $X, Y$ respectively, and suppose the following inequalities hold,
\begin{itemize}
    \item $s[X_1+X_2; Y_1+Y_2] \leq \eta(\H[X_1+X_2] + \H[Y_1+Y_2]) + \kappa$,
    \item $s[X_1+Y_2; X_2+Y_1] \leq \eta(\H[X_1+Y_2] + \H[X_2+Y_1]) + \kappa$,
    \item $s[X_1 | X_1+X_2 ; Y_1 | Y_1+Y_2] \leq \eta (\H[X_1 | X_1+X_2] + \H[Y_1 | Y_1+Y_2]) + \kappa$, 
    \item $s[X_1 | X_1+Y_2; Y_1 | Y_1+X_2] \leq \eta (\H[X_1 | X_1+Y_2] + \H[Y_1 | Y_1+X_2]) + \kappa$.
\end{itemize}
For  $(u, w) \in \supp (X_1+Y_2) \times \supp(Y_1+X_2)$, let 
\[
X_u = (X_1 | X_1+Y_2 = u)\quad\text{and}\quad Y_w = (Y_1 | Y_1+X_2 = w).
\]
For each $(u,w)$ there is a subspace $V(u, w) \subset G$ such that $\H[\U_{V(u,w)}] \leq 7(\H[X_u] + \H[Y_u])$ and 
\[
\E_{u \sim X_1+Y_2, v \sim Y_1+X_2} (\H[\pi_{V(u,w)}(X_u)] + \H[\pi_{V(u,w)}(Y_w)]) \leq 480 \kappa.
\]
\end{lemma}

\begin{proof}
Let $Z_1,Z_2,Z_3,S$ be as in~\eqref{eq:Zdefn}.
First we sum the left-hand sides of the equations in the Lemma statement using~\eqref{eq:FirstFibring} and~\eqref{eq:SecondFibring},
\begin{align}
\text{Sum of left-hand sides} &= 4s[X ; Y]+ \I[Z_1 : Z_3 | S] + \I[Z_1 : Z_2 | S] \nonumber \\ 
&\geq 4\eta(\H[X] + \H[Y])+ \I[Z_1 : Z_3 | S] + \I[Z_1 : Z_2 | S].\nonumber
\end{align}
For the right-hand sides, we pair together terms using
\begin{align*}
\H[X_1 | X_1 + X_2] + \H[X_1 + X_2] &= \H[(X_1, X_2) | X_1+X_2] + \H[X_1+X_2] \\ 
&= \H[X_1,X_2] = 2\H[X],
\end{align*}
and similar inequalities, to get
\begin{align*}
    \text{Sum of right-hand sides} &= 4\eta( \H[X] + \H[Y]) + 4\kappa. 
\end{align*}
When we compare both sides, the $4\eta (\H[X] + \H[Y])$ terms cancel to give
\[
\I[Z_1 : Z_3 | S] + \I[Z_1 : Z_2 | S]  \leq 4\kappa.
\]
Because $X_1$ and $X_2$ play symmetric roles, $\I[Z_1 : Z_3 | S] = \I[Z_2 : Z_3 | S]$, thus 
\[
\I[Z_i : Z_j | S] \leq 4\kappa \qquad \text{for all $i \neq j$.}
\]
Using mutual information, we can compare the conditional entropy of $Z_1$ and $Z_3$,
\[
\H[Z_3 | S] = \H[Z_1+Z_2 | S] \geq \H[Z_1 | S, Z_2] \stackrel{\eqref{eq:MutualInformationConditionalEntropy}}{=} \H[Z_1 | S] - \I[Z_1 : Z_2 | S] \geq \H[Z_1 | S] - 4\kappa. 
\]
The same inequality holds for any pair, so 
\[
\big|\, \H[Z_i | S] - \H[Z_j | S]\, \big| \leq 4\kappa\qquad \text{for $i,j \in \{1,2,3\}$.}
\] 
Using~\Cref{lem:BSG} and the fact that $Z_1+Z_2 = Z_3$,
\begin{equation}\label{eq:ExpectedDistance}
\begin{split}
\E_{(t,r ) \sim (Z_2, S)} \d[(Z_1 | Z_2 = t, S = r); (Z_3 | Z_2 = t, S = r)] &\leq 3 \I[Z_1 : Z_3 | S] + 2\H[Z_2 | S] - \H[Z_1 | S] - \H[Z_3 | S] \\
&\leq 20\kappa.    
\end{split}
\end{equation}
Unpacking the definition of $Z_1,Z_2,Z_3,S$ lets us write this equation in terms of $X_1,X_2,Y_1,Y_2$,
\begin{align*}
\d[(Z_1 | Z_2 = t, S = r); (Z_3 | Z_2 = t, S = r)]  = \d[(X_1+Y_1 | X_2+Y_1 = t, X_1+Y_2 = r - t); \\ 
(X_1+X_2 | X_2+Y_1 = t, X_1+Y_2 = r-t)].
\end{align*}
Let 
\[
U = X_1+Y_2,\; W = Y_1+X_2,\qquad X_u = (X_1 | U = u),\; Y_w = (Y_1 | W = w). 
\]
The first random variable in the Ruzsa distance above has the same distribution as $X_{r-t} + Y_t$. 
By~\Cref{cor:KillingSubspacePFR}, for each $(u, w) \in \supp U \times \supp W$, there is a subspace $V(u,w) \subset G$ with $\H[\U_{V(u,w)}] \leq 7(\H[X_u] + \H[Y_w])$ such that 
\begin{align*}
\H[\pi_{V(u,w)}(X_u + Y_w)] &\leq 12\d[X_u + Y_w; (X_1+X_2 | X_2+Y_1 = w, X_1+Y_2 = u)].
\end{align*}
We use the estimate
\[
\H[\pi_{V(u,w)}(X_u)] + \H[\pi_{V(u,w)}(Y_w)]  \leq 2\H[\pi_{V(u,w)}(X_u + Y_w)].
\]
Under the change of variable $(t, r) \to (r-t, t)$ the expected value $\E_{(t,r ) \sim (Z_2, S)}$ is the same as $\E_{(u, w) \sim (U, W)}$, so~\eqref{eq:ExpectedDistance} gives
\[
\E_{u \sim U, w \sim W}\bigl( \H[\pi_{V(u,w)}(X_u)] + \H[\pi_{V(u,w)}(Y_w)]\bigr) \leq 480\, \kappa.
\]
\end{proof}

\section{Proof setup}\label{sec:setup-induction}

The proof of Theorem \ref{thm:main-entropy} proceeds by induction on the doubling mass of $X$ and $Y$. Below we define two intermediate statements which we will induct on. 

\begin{definition}\leavevmode
\begin{itemize}
    \item For $\eta \in (0, 1/2]$, $L\geq 0$, and $c \in (0,1)$, let $\mb A(\eta, L , c)$ denote the following assertion.
    For any independent random variables $X, Y$ on $G$ such that $\H[X+Y] \leq (1 - \eta) (\H[X]+\H[Y])$, there is a subspace $V$ with $\H[\U_V] \leq L (\H[X]+\H[Y])$ and 
    \[
    \H[\pi_{V}(X)] + \H[\pi_{V}(Y)] \leq (1 - c) (\H[X] + \H[Y]). 
    \]

    \item For $\eta \in (0, 1/2]$, $\varepsilon \in (0,1]$, and $L \geq 0$, let $\mb B(\eta, \varepsilon, L)$ denote the following assertion.
    For any independent random variables $X, Y$ on $G$, there is a subspace $V$ with $\H[\U_V] \leq L (\H[X] + \H[Y])$ and 
    \[
    \H[\pi_{V}(X) + \pi_{V}(Y)] \geq (1 - \eta) (\H[\pi_{V}(X)] + \H[\pi_{V}(Y)]) - \varepsilon (\H[X] + \H[Y]). 
    \]
\end{itemize}
\end{definition}

Our goal is to show that for all $\eta \in (0, 1/2]$ and $\varepsilon > 0$, $\mb B(\eta, \varepsilon, L)$ holds for some $L$. 

\begin{lemma}\label{lem:BasicRelationsAB}
\leavevmode
\begin{enumerate}[label=(\alph*)]
    \item If $\eta' > \eta + \varepsilon$, then $\mb B(\eta, \varepsilon, L) \Longrightarrow \mb A(\eta', L, \eta' - \eta - \varepsilon)$.
    \item If $\varepsilon \in (0,1]$, then $\mb A(\eta, L, c) \Longrightarrow \mb B(\eta, \varepsilon, \lceil \frac{\log \varepsilon^{-1}}{c} \rceil L)$.
\end{enumerate}
\end{lemma}
\begin{proof}
\textbf{(a)}
Suppose $\H[X+Y] \leq (1 - \eta') (\H[X] + \H[Y])$. 
Let $V$ be a subspace with $\H[\U_V] \leq L (\H[X]+\H[Y])$ and 
\[
\H[\pi_{V}(X) + \pi_{V}(Y)] \geq (1-\eta) (\H[\pi_{V}(X)] + \H[\pi_{V}(Y)]) - \varepsilon(\H[X] + \H[Y]).
\]
Let $\delta > 0$ represent the amount of entropy killed under projection,
\[
\H[\pi_{V}(X) ] + \H[\pi_{V}(Y)] = (1 - \delta)(\H[X] + \H[Y]),
\]
so that
\[
\H[\pi_{V}(X) + \pi_{V}(Y)] \geq ((1 - \eta)(1-\delta) - \varepsilon) (\H[X] + \H[Y]). 
\]
The trivial upper bound 
\[
\H[\pi_{V}(X) + \pi_{V}(Y)] \leq \H[X+Y] \leq (1 - \eta')(\H[X]+\H[Y])
\]
implies 
\[
(1 - \eta)(1-\delta) - \varepsilon  \leq 1 - \eta',
\]
and rearranging yields 
\[
\delta \geq \frac{\eta' - \eta - \varepsilon}{1 - \eta} \geq \eta' - \eta - \varepsilon.
\]

\bigskip \noindent 
\textbf{(b)}
Fix $\varepsilon > 0$. 
Let $X, Y$ be independent random variables on $G$. Define 
\[
G_0 = G,\quad X_0 = X,\quad Y_0 = Y,
\]
and construct a sequence $(X_i, Y_i)$ as follows. If $\H[X_i+Y_i] \leq (1 - \eta - \varepsilon) (\H[X_i] + \H[Y_i])$ then apply $\mb A(\eta, L, c)$ to produce a subspace $V_i$. Let 
\[
G_{i+1} = G_i / V_i,\quad X_{i+1} = \pi_{G_i/V_i}(X_i),\quad Y_{i+1} = \pi_{G_i/V_i}(Y_i)
\]
so that 
\[
\H[X_{i+1}] + \H[Y_{i+1}] \leq (1 - c) (\H[X_i] + \H[Y_i]). 
\]
After $j$ steps, we have quotiented $G_0$ by a subspace of dimension $\leq j L$. 
Let $m = \lceil \frac{\log \varepsilon^{-1}}{c}\rceil$, so that  $(1 - c)^m \leq \varepsilon$. If $X_m$ and $Y_m$ are defined, meaning the process did not terminate already, then 
\[
\H[X_m] + \H[Y_m] \leq \varepsilon (\H[X] + \H[Y])
\]
which implies 
\[
\H[X_m + Y_m] \geq (1 - \eta) (\H[X_m] + \H[Y_m]) - \varepsilon(\H[X] + \H[Y])
\]
so the conclusion $\mb B(\eta, \varepsilon, mL)$ holds. 

If the process terminates at step $i < m$, that means 
\[
\H[X_i+Y_i] \geq ( 1- \eta - \varepsilon) (\H[X_i] + \H[Y_i]) \geq (1 - \eta)(\H[X_i] + \H[Y_i]) - \varepsilon (\H[X] + \H[Y])
\]
as desired. 
\end{proof}

\section{Inductive step}\label{sec:inductive-step}

Let $X$ and $Y$ be random variables and $X_j, Y_j$ be independent copies. 
In the following lemma, we apply $\mb B$ several times to remove additive structure between $(X_1+X_2, Y_1+Y_2)$ and between $(X_1+Y_2, Y_1+X_2)$. 
\begin{lemma}\label{lem:MakeSumsetsNotDouble}
Assume $\mb B(\eta_0, \varepsilon_0, L_0)$ holds.
For any $G$-valued independent random variables $X$ and $Y$, there exists a subspace $V$ with $\H[\U_V] \leq \frac{4}{\varepsilon_0} L_0$ such that, letting $\pi = \pi_{V}$,
\small
\begin{align}
\H[\pi(X_1+X_2+Y_1+Y_2)] &\geq (1 - \eta_0) (\H[\pi(X_1+X_2)] + \H[\pi(Y_1+Y_2)]) -  4 \varepsilon_0  (\H[X]+\H[Y]),\label{eq:SumsetsDouble1} \\ 
\H[\pi(X_1+X_2+Y_1+Y_2)] &\geq (1 - \eta_0) (\H[\pi(X_1+Y_2)] + \H[\pi(Y_1+X_2)]) -  4 \varepsilon_0  (\H[X]+\H[Y]).\label{eq:SumsetsDouble2}
\end{align}
\normalsize
\end{lemma}
\begin{proof}
Start by setting 
\[
X^{(0)} = X,\quad Y^{(0)} = Y,\quad V^{(0)} = \{0\}. 
\]
Assume~\eqref{eq:SumsetsDouble1} fails, so
\begin{align}\label{eq:AssumedFailed}
\H[X_1+X_2+Y_1+Y_2] \leq (1 - \eta_0) (\H[X_1+X_2] + \H[Y_1+Y_2]) -  4 \varepsilon_0 (\H[X]+\H[Y]).
\end{align}
Apply $\mb B(\eta_0, \varepsilon_0, L_0)$ to $X_1+X_2$ and $Y_1+Y_2$ to produce a subspace $V^{(1)}$ with 
\[
\H[\U_{V^{(1)}}] \leq L_0 (\H[X_1 + X_2] + \H[Y_1 + Y_2]) \leq 2L_0 (\H[X] + \H[Y])
\]
such that, for
\[
X^{(1)} = \pi_{V^{(1)}}(X),\quad Y^{(1)} = \pi_{V^{(1)}}(Y),
\]
we have
\begin{align}
\H[X^{(1)}_1 + X^{(1)}_2 + Y^{(1)}_1 + Y^{(1)}_2] &\geq (1 - \eta_0) (\H[X^{(1)}_1 + X^{(1)}_2] + \H[Y^{(1)}_1 + Y^{(1)}_2]) -  \varepsilon_0 (\H[X_1 + X_2]+\H[Y_1 + Y_2]), \nonumber \\ 
&\geq (1 - \eta_0) (\H[X^{(1)}_1 + X^{(1)}_2] + \H[Y^{(1)}_1 + Y^{(1)}_2]) -  2\varepsilon_0 (\H[X]+\H[Y]).\label{eq:StructAfterApplyingB}
\end{align}
Comparing~\eqref{eq:AssumedFailed} to~\eqref{eq:StructAfterApplyingB} gives
\begin{align*}
\H[X^{(1)}_1 + X^{(1)}_2] + \H[Y^{(1)}_1 + Y^{(1)}_2]  &\leq \H[X_1+X_2] + \H[Y_1+Y_2] -  (4 \varepsilon_0 - 2\varepsilon_0) (\H[X]+\H[Y]).
\end{align*}
If~\eqref{eq:SumsetsDouble2} fails instead, we apply $\mb B(\eta_0, \varepsilon_0, L_0)$ to $X_1+Y_2$ and $Y_1+X_2$. It is possible that we apply $\mb B$ to fix~\eqref{eq:SumsetsDouble1}, but the new variables fail~\eqref{eq:SumsetsDouble2}, so we have to iterate this procedure. We produce a sequence 
\[
X^{(j)} = \pi_{V^{(j)}}(X),\quad Y^{(j)} = \pi_{V^{(j)}}(Y),\quad V^{(j)} \subset G
\]
where $\H[\U_{V^{(j)}}] \leq 2jL_0\, (\H[X]+\H[Y])$. If we use $\mb B$ to fix~\eqref{eq:SumsetsDouble1} at least $j/2$ times, then
\[
\H[X^{(j)}_1 + X^{(j)}_2] + \H[Y_1^{(j)} + Y_2^{(j)}]  \leq \H[X_1+X_2] + \H[Y_1+Y_2] - \frac{j}{2}(2\varepsilon_0) (\H[X]+\H[Y]),
\]
and if we use it to fix~\eqref{eq:SumsetsDouble1} at least $j/2$ times, 
\[
\H[X^{(j)}_1+Y^{(j)}_2] + \H[X^{(j)}_2 + Y^{(j)}_1] \leq \H[X_1+Y_2] + \H[X_2+Y_1] -  \frac{j}{2}(2\varepsilon_0) (\H[X]+\H[Y]).
\]
If 
\[
\frac{j}{2}(2\varepsilon_0) > 2
\]
then the right hand side is negative, thus this pocess must terminate after at most $\frac{2}{\varepsilon_0}$ steps. The final subspace has dimension $\leq \frac{4}{\varepsilon_0} L_0 (\H[X] + \H[Y])$. 
\normalsize
\end{proof}

The next lemma roughly says that if $X$ and $Y$ can be decomposed into pieces, most pairs of which additively interact with each other through a subspace, then there is one big subspace that removes a substantial amount of entropy from $Y$. 
\begin{lemma}[Local to global structure]\label{lem:LocalToGlobal}
Let $X$ and $Y$ be independent $G$-valued random variables. Suppose we are given couplings $(X, U)$ and $(Y, W)$, where the only dependencies are between $(X, U)$ and $(Y, W)$. For $u \in \supp U$ and $w \in \supp W$, we consider the conditioned random variables
\[
X_u := (X | U = u),\quad Y_w := (Y | W = w). 
\]
Suppose that for every $u \in \supp U$ and $w \in \supp W$ we are given a subspace $V(u,w)$ such that 
\begin{equation}\label{eq:LocalStructureHyp}
\E_{u\sim U, w\sim W}\, \s[X_u | \pi_{V(u,w)}(X_u); Y_w | \pi_{V(u,w)}(Y_w)] \geq \zeta (\H[X] + \H[Y]).
\end{equation}
Then there exists a subspace $\bar V$ with 
\[
\H[\U_{\bar V}] \leq \frac{8}{\zeta^2} \E_{u,w} \H[\U_{V(u,w)}]
\]
and
\[
\H[Y | \pi_{\bar V}(Y)] \geq \frac{\zeta}{4} (\H[X] + \H[Y]). 
\]
\end{lemma}

Here is the proof idea. Start by fixing $u \in \supp U$, and suppose that $V_u$ is a subspace that removes a substantial amount of entropy from $X_u$, and that larger subspaces do not remove much more entropy. For typical $w \sim W$, we would like to say that $V_u$ removes substantial entropy from $Y_w$. To prove this, we consider the diagram 
\begin{equation}\label{eq:SequenceOfSubspacesDiagram}
\begin{tikzcd}
    X_u \dar & Y_w \dar \\ 
    X_u \mod V_u \cap V(u,w) \dar & Y_w \mod V_u \cap V(u,w) \dar\\ 
    X_u \mod V(u,w)  & Y_w \mod V(u,w) 
\end{tikzcd}
\end{equation}
Our assumption tells us that $X_u$ and $Y_w$ additively interact through the subspace $V(u,w)$. 
By subspace submodularity, Eq.~\eqref{eq:SubspaceSubmodularity},
\[
\H[\pi_{V_u \cap V(u,w)}(X_u)] - \H[\pi_{V(u,w)}(X_u)] \leq \H[\pi_{V(u,w)}(X_u)] - \H[\pi_{V_u + V(u,w)}(X_u)],
\]
and the maximality assumption on $V_u$ tells us the right hand side is small. It follows that there is not much additive interaction between $X_u$ and $Y_w$ in $V(u,w) \bmod V_u \cap V(u,w)$ (corresponding to the bottom arrows). By~\eqref{eq:FiberInteraction}, $X_u$ and $Y_u$ must additively interact through the subspace $V_u \cap V(u,w)$ (corresponding to the top arrows). That implies $V_u \cap V(u,w)$ removes a significant amount of entropy from $Y_w$. Because this is true for typical $Y_w$, the one subspace $V_u$ removes a significant amount of entropy from $Y$. 

In order to execute this argument, we judiciously choose an integer $k$, select a random $u \sim U$ and an independent random sequence $w^{(1)} \sim W, \ldots, w^{(k)} \sim W$, and set 
\[
\bar V = V(u, w^{(1)}) + \dots + V(u, w^{(k)}). 
\]
The integer $k$ is chosen so that if we added another random subspace on the right hand side, it would not remove much more entropy from $X_u$. 

Before proving~\Cref{lem:LocalToGlobal}, we state a lemma related to diagram~\eqref{eq:SequenceOfSubspacesDiagram}. The lemma says, if $X$ and $Y$ additvely interact through a subspace $V$, and if $W \subset V$ is another subspace such that not much entropy is killed going from $(X \bmod W)$ to $(X \bmod V)$, then entropy must be killed going from $Y$ to $(Y \bmod W)$. 
\begin{lemma}\label{lem:YSizeLowerBd}
Let $X$ and $Y$ be independent $G$-valued random variables. Let $W \subset V$ be subspaces. Then 
\[
\H[Y | \pi_W(Y)] \geq \s[X | \pi_V(X) ; Y | \pi_V(Y)] - \H[\pi_W(X) | \pi_V(X)]. 
\]
\end{lemma}
\begin{proof}
By~\eqref{eq:FiberInteraction},
\begin{align*}
    \s[X | \pi_V(X) ; Y | \pi_V(Y)] &\leq \s[X | \pi_W(X) ; Y | \pi_W(Y)] + \s[\pi_W(X) | \pi_V(X) ; \pi_W(Y) | \pi_V(Y)].
\end{align*}
By~\eqref{eq:TrivialEstimateDoublingMass}, 
\begin{align*}
    \s[X | \pi_W(X) ; Y | \pi_W(Y)] &\leq \H[Y | \pi_W(Y)], \\ 
    \s[\pi_W(X) | \pi_V(X) ; \pi_W(Y) | \pi_V(Y)] &\leq \H[\pi_W(X) | \pi_V(X)].
\end{align*}
The result follows.
\end{proof}

\begin{proof}[Proof of~\Cref{lem:LocalToGlobal}]
Let $u \in \supp U$ and $w^{(1)}, w^{(2)}, \ldots \in \supp W$ be a sequence of points .
For $j \geq 1$, we let
\begin{align*}
V_j &= V(u, w^{(j)}), \\ 
V_{\leq j} &= V_1 + \dots + V_j.
\end{align*}
These subspaces depend on the sequence at hand. 

We set 
\[
h_j := \E_{u \sim U} \E_{w^{(1)}, \ldots, w^{(k)} \sim W} \H[\pi_{V_{\leq j}}(X_u)].
\]
Note that $h_0 = \H[X | U] = \E_{u\sim U} \H[X_u] \leq \H[X]$. For any fixed sequence $(u, w^{(1)}, \ldots)$ the entropy $\H[\pi_{V_{\leq j}}(X_u)]$ is monotonically decreasing in $j$, so $h_j$ is monotonically decreasing as well. Fix a constant $\tau > 0$ to be chosen later. By the pigeonhole principle, there is some $k \leq \tau^{-1}$ for which 
\[
h_k - h_{k+1} \leq \tau h_0. 
\]

Fix a sequence $u, w^{(1)}, \ldots, w^{(k+1)}$. 
By~\Cref{lem:YSizeLowerBd}, 
\begin{align*}
    \H[Y_{w^{(k+1)}} | \pi_{V_{k+1} \cap V_{\leq k}}(Y_{w^{(k+1)}})] &\geq \s[X_u | \pi_{V_{ k+1}}(X_u) ; Y_{w^{(k+1)}} | \pi_{V_{ k+1}}(Y_{w^{(k+1)}})] - \H[\pi_{V_{k+1} \cap V_{\leq k}}(X_u) | \pi_{V_{k+1}}(X_u)] \\ 
    &\stackrel{\eqref{eq:SubspaceSubmodularity}}{\geq} \s[X_u | \pi_{V_{ k+1}}(X_u) ; Y_{w^{(k+1)}} | \pi_{V_{ k+1}}(Y_{w^{(k+1)}})] - \H[\pi_{V_{\leq k}}(X_u) | \pi_{V_{\leq k+1}}(X_u)].
\end{align*}
Now fix $u, w^{(1)}, \ldots, w^{(k)}$ and take an expected value over $w^{(k+1)}$ to find, by submodularity, 
\begin{align*}
\H[Y | \pi_{V_{\leq k}}(Y)] &\geq \H[Y | \pi_{V_{k+1} \cap V_{\leq k}}(Y)] \\ 
&\geq  \E_{w^{(k+1)} \sim W} \H[Y_{w^{(k+1)}} | \pi_{V_{k+1} \cap V_{\leq k}}(Y_{w^{(k+1)}})]  \\ 
&\geq \E_{w^{(k+1)}} \bigl(\s[X_u | \pi_{V_{ k+1}}(X_u) ; Y_{w^{(k+1)}} | \pi_{V_{ k+1}}(Y_{w^{(k+1)}})] - \E_{w^{(k+1)}}  \H[\pi_{V_{\leq k}}(X_u) | \pi_{V_{\leq k+1}}(X_u)]\bigr).
\end{align*}
Average over $u, w^{(1)}, \ldots, w^{(k)}$ to find 
\begin{align*}
\E_{u, w^{(1)}, \ldots, w^{(k)}} \H[Y | \pi_{V_{\leq k}}(Y)] &\geq \E_{u, w^{(k+1)}} \s[X_u | \pi_{V_{ k+1}}(X_u) ; Y_{w^{(k+1)}} | \pi_{V_{ k+1}}(Y_{w^{(k+1)}})] - \\ 
&\qquad \E_{u, w^{(1)}, \ldots, w^{(k+1)}}  \H[\pi_{V_{\leq k}}(X_u) | \pi_{V_{\leq k+1}}(X_u)] \\ 
&= \E_{u,w} \s[X_u | \pi_{V(u,w)}(X_u) ; Y_w | \pi_{V(u,w)}(Y_w)] - (h_{k} - h_{k+1}) \\ 
&\geq \zeta(\H[X] + \H[Y]) - \tau \H[X] \geq (\zeta - \tau)(\H[X] + \H[Y]). 
\end{align*}
Choose $\tau = \zeta/2$. 
For any subspace $V$, 
\[
\H[Y | \pi_{V}(Y)] \leq \H[Y], 
\]
so 
\begin{align*}
\frac{\zeta}{2}(\H[X] + \H[Y]) &\leq \E_{u, w^{(1)}, \ldots, w^{(k)}} \H[Y | \pi_{V_{\leq k}}(Y)] \\ 
&\leq \frac{\zeta}{4}(\H[X] + \H[Y]) + \H[Y] \Pr\bigl[\,\H[Y | \pi_{V_{\leq k}}(Y)] \geq \frac{\zeta}{4}(\H[X] + \H[Y])  \, \bigr]
\end{align*}
and thus 
\[
\Pr\bigl[\,\H[Y | \pi_{V_{\leq k}}(Y)] \geq \frac{\zeta}{4}(\H[X] + \H[Y])  \, \bigr] \geq \frac{\zeta}{4}. 
\]
To estimate the size of $V_{\leq k}$, use linearity of expectation,
\[
\E_{u, w^{(1)}, \ldots, w^{(k)}} \H[\U_{V_{\leq k}}] \leq \E_{u, w^{(1)}, \ldots, w^{(k)}}(\sum_{i \leq k} \H[\U_{V_{i}}] ) \leq k \, \E_{u,w}\H[\U_{V(u,w)}],
\]
and by Markov's inequality 
\[
\Pr[\H[\U_{V_{\leq k}}]  > \frac{4}{\zeta} k \, \E_{u,w}\H[\U_{V(u,w)}] ] < \frac{\zeta}{4}.
\]
Thus there is some choice of $u, w^{(1)}, \ldots, w^{(k)}$ for which 
\begin{equation}\label{eq:LowerBdYFibers}
\H[Y | \pi_{V_{\leq k}}(Y)] \geq \frac{\zeta}{4}(\H[X] + \H[Y])
\end{equation}
and 
\[
\H[\U_{V_{\leq k}}] \leq \frac{4}{\zeta} k \E_{u,w}\H[\U_{V(u,w)}] \leq \frac{8}{\zeta^2} \E_{u,w}\H[\U_{V(u,w)}].
\]
\end{proof}

We are ready to prove the inductive step.
\begin{proposition}\label{prop:InductiveStep}
There exists a constant $C > 0$ such that, for $\eta_0 \in (0, 1/2]$,
\[
\mb B\bigl(\eta_0,\; C^{-1} \eta_0^2,\; L_0\bigr) \Longrightarrow \mb A\bigl(\eta_0 - C^{-1} \eta_0^2,\;  C^3 \eta_0^{-4} \max\{L_0, 1\},\; C^{-1} \eta_0^2\bigr). 
\]
In fact, we may take $C = 2^{15}$. 
\end{proposition}
\begin{proof}
Assume $\mb B(\eta_0, \varepsilon_0, L_0)$ holds. 
Our goal is to prove $\mb A(\eta_0 - \varepsilon_0, L_1, c)$ holds for some $L_1 > 0$ and $c > 0$. 
We will need to assume $\varepsilon_0 \leq 2^{-15} \eta_0^2$. The proof will give
\begin{equation}\label{eq:OverallBounds}
\begin{split}
    L_1 &= \max \{ 12 \varepsilon_0^{-2} L_0, 2^{12} \eta_0^{-4} \}, \\ 
    c &= \min\{\varepsilon_0, \frac{\eta_0^2}{32}\}.
\end{split}
\end{equation}
In the statement of~\Cref{prop:InductiveStep}, we specialize to $\varepsilon_0 = C^{-1} \eta_0^2$ for $C = 2^{15}$. 

Let $\tilde X$ and $\tilde Y$ be random variables. As in the hypotheses of $\mb A(\eta_0 - \varepsilon_0, \bullet, \bullet)$, assume $\s[\tilde X ; \tilde Y] \geq (\eta_0 - \varepsilon_0) (\H[\tilde X] + \H[\tilde Y])$. 

Start by applying~\Cref{lem:MakeSumsetsNotDouble} to obtain a subspace $V_0$ with $\H[\U_{V_0}] \leq \frac{4}{\varepsilon_0} L_0$ such that, letting
\[
X = \pi_{V_0}(\tilde X)\quad\text{and}\quad Y = \pi_{V_0}(\tilde Y),
\]
we have
\begin{align*}
\s[X_1+X_2; Y_1+Y_2] &\leq \eta_0 (\H[X_1+X_2] + \H[Y_1+Y_2]) +  4 \varepsilon_0 (\H[\tilde X] + \H[\tilde Y]),\\ 
\s[X_1+Y_2 ; X_2 + Y_1] &\leq \eta_0 (\H[X_1+Y_2] + \H[X_2+Y_1]) +  4\varepsilon_0  (\H[\tilde X] + \H[\tilde Y]).
\end{align*}
We may assume 
\[
\H[X] + \H[Y] \geq (1-c) (\H[\tilde X] + \H[\tilde Y]),
\]
because otherwise we are done. Because $c \leq \frac{1}{2}$, 
\begin{equation}\label{eq:NonDoubling}
\begin{split}
\s[X_1+X_2; Y_1+Y_2] &\leq \eta_0 (\H[X_1+X_2] + \H[Y_1+Y_2]) +  8\varepsilon_0(\H[X] + \H[Y]),\\ 
\s[X_1+Y_2 ; X_2 + Y_1] &\leq \eta_0 (\H[X_1+Y_2] + \H[X_2+Y_1]) +  8\varepsilon_0(\H[X] + \H[Y]).    
\end{split}
\end{equation}
We can also get a lower bound on $\s[X ; Y]$. By~\eqref{eq:FibringIdentity}, 
\begin{align*}
\s[X ; Y] &\geq \s[\tilde X; \tilde Y] - \s[\tilde X | \pi_{V_0}(\tilde X) ; \tilde Y | \pi_{V_0}(\tilde Y)]  \\ 
&\geq \s[\tilde X; \tilde Y] - \H[\tilde X | \pi_{V_0}(\tilde X)] - \H[ \tilde Y | \pi_{V_0}(\tilde Y)] \\ 
&\geq \s[\tilde X; \tilde Y] - c (\H[\tilde X] + \H[\tilde Y]) \\
&\geq (\eta_0 - \varepsilon_0 - c) \H[\tilde X ; \tilde Y] \geq (\eta_0 - 2\varepsilon_0) \H[\tilde X ; \tilde Y]
\end{align*}
where in the last inequality we assume $c \leq \varepsilon_0$. 

At this point we split into three cases, depending on the doubling mass in the fibers. 

\bigskip
\noindent \textbf{Case 1: $\s[X_1 | X_1 + X_2 ; Y_1 | Y_1 + Y_2] \geq \eta_0 (\H[X_1 | X_1 + X_2] + \H[Y_1 | Y_1 + Y_2]) + 8\varepsilon_0(\H[X] + \H[Y]) $.}

We apply~\Cref{lem:LocalToGlobal} (local to global structure), with 
\[
X = X_1,\; U = X_1+X_2,\quad Y = Y_1,\; W = Y_1+Y_2
\]
and we use the notation
\[
X_u = (X | U = u),\quad Y_w = (Y | W = w). 
\]
For $(u, w) \in \supp U \times \supp W$, define $V(u, w)$ to be the subspace produced by applying $\mb B(\eta_0, \varepsilon_0, L_0)$ to $X_u, Y_w$. We are guaranteed that 
\[
\H[\U_{V(u,w)}] \leq L_0 (\H[X_u] + \H[Y_w])
\]
and 
\[
\s[\pi_{V(u,w)}(X_u) ; \pi_{V(u,w)}(Y_w)] \leq \eta_0 (\H[\pi_{V(u,w)}(X_u)] + \H[\pi_{V(u,w)}(Y_w)]) + \varepsilon_0 (\H[X_u] + \H[Y_w]). 
\]
By~\eqref{eq:FibringIdentity}, 
\begin{align*}
\s[X_u | \pi_{V(u,w)}(X_u) ; Y_w | \pi_{V(u,w)}(Y_w)] &\geq s[X_u ; Y_w] - \s[\pi_{V(u,w)}(X_u) ; \pi_{V(u,w)}(Y_w)] \\ 
&\geq s[X_u ; Y_w] - \eta_0 (\H[\pi_{V(u,w)}(X_u)] + \H[\pi_{V(u,w)}(Y_w)]) - \varepsilon_0 (\H[X_u] + \H[Y_w]).
\end{align*}
Taking an expected value over $u\sim U$ and $w\sim W$, and using the Case 1 assumption, gives
\[
\E_{u,w}\s[X_u | \pi_{V(u,w)}(X_u) ; Y_w | \pi_{V(u,w)}(Y_w)] \geq 7\varepsilon_0 (\H[X] + \H[Y]).
\]
By~\Cref{lem:LocalToGlobal}, there exists a subspace $\bar V$ with
\[
\H[\U_{\bar V}] \leq 8 \varepsilon_0^{-2} \E_{u,w} L_0 (\H[X_u] + \H[Y_w]) \leq 8 \varepsilon_0^{-2} L_0 (\H[X] + \H[Y]),
\]
such that 
\[
\H[\pi_{\bar V}(X)] + \H[\pi_{\bar V}(Y)] \leq ( 1 - \varepsilon_0)(\H[X] + \H[Y]) \leq ( 1 - \varepsilon_0)(\H[\tilde X] + \H[\tilde Y]).
\]
Thus $V_0 + \bar V$ satisfies the conclusion with
\begin{equation}\label{eq:Case1Bounds}
\begin{split}
    L_1 &= 8\varepsilon_0^{-2} L_0 + 4\varepsilon_0^{-1} L_0 \leq 12\varepsilon_0^{-2} L_0, \\ 
    c &= \varepsilon_0. 
\end{split}
\end{equation}

\bigskip
\noindent \textbf{Case 2: $\s[X_1 | X_1 + Y_2 ; Y_1 | Y_1 + X_2] \geq \eta_0 (\H[X_1 | X_1 + Y_2] + \H[Y_1 | Y_1 + X_2]) + 8\varepsilon_0(\H[X] + \H[Y]) $.}
This case is identical to the prior case, but with $U = X_1 + Y_2$ and $W = Y_1 + X_2$ instead.

\bigskip
\noindent \textbf{Case 3: Endgame.}
Recall
\[
\H[X+Y] \geq (\eta_0 - 2\varepsilon_0) (\H[X] + \H[Y]). 
\]

Combine~\eqref{eq:NonDoubling} with the assumption that we are not in Case 1 or Case 2 to obtain the four endgame inequalities,
\begin{align*}
s[X_1+X_2; Y_1+Y_2]
&\leq (\eta_0-2\varepsilon_0)\bigl(\H[X_1+X_2] + \H[Y_1+Y_2]\bigr)
     + 12\varepsilon_0 \bigl(\H[X] + \H[Y]\bigr), \\
s[X_1+Y_2; X_2+Y_1]
&\leq (\eta_0-2\varepsilon_0)\bigl(\H[X_1+Y_2] + \H[X_2+Y_1]\bigr)
     + 12\varepsilon_0 \bigl(\H[X] + \H[Y]\bigr), \\
s[X_1 \mid X_1+X_2 ; Y_1 \mid Y_1+Y_2]
&\leq (\eta_0-2\varepsilon_0)\bigl(\H[X_1 \mid X_1+X_2]
     + \H[Y_1 \mid Y_1+Y_2]\bigr)
     + 12\varepsilon_0 \bigl(\H[X] + \H[Y]\bigr), \\
s[X_1 \mid X_1+Y_2 ; Y_1 \mid Y_1+X_2]
&\leq (\eta_0-2\varepsilon_0)\bigl(\H[X_1 \mid X_1+Y_2]
     + \H[Y_1 \mid Y_1+X_2]\bigr)
     + 12\varepsilon_0 \bigl(\H[X] + \H[Y]\bigr).
\end{align*}

Let 
\[
U = X_1 + Y_2,\quad W = Y_1+X_2,\qquad X_u = (X | U = u),\quad Y_w = (Y | W = w).
\]
We apply the endgame~(\Cref{lem:Endgame}) and find that for every $(u, w) \in \supp U \times \supp W$, there is a subspace $V(u, w)$ with $\H[\U_{V(u,w)}] \leq 7(\H[X_u] + \H[Y_w])$, such that
\begin{align*}
\E_{u,w} (\H[\pi_{V(u,w)}(X_u)] + \H[\pi_{V(u,w)}(Y_w)]) &\leq 480\cdot 12\, \varepsilon_0 \bigl(\H[X] + \H[Y]\bigr) \\
&\leq 2^{13} \varepsilon_0 \bigl(\H[X] + \H[Y]\bigr).
\end{align*}
By~\eqref{eq:FibringIdentity},
\begin{align*}
\E_{u,w} \s[X_u | \pi_{V(u,w)}(X_u) ; Y_w | \pi_{V(u,w)}(Y_w)] &\geq \E_{u,w} \s[X_u  ; Y_w] - \E_{u,w} \s[\pi_{V(u,w)}(X_u) ; \pi_{V(u,w)}(Y_w)] \\ 
&\geq \E_{u,w} \s[X_u  ; Y_w] - 2^{13} \varepsilon_0 \bigl(\H[X] + \H[Y]\bigr). 
\end{align*}
To estimate the first term, use~\eqref{eq:SecondFibring} to relate
\begin{align*}
\s[X_1 + Y_2 ; Y_1 + X_2] + \E_{u,w} \s[X_u  ; Y_w] &= \s[X_1 + Y_2 ; Y_1 + X_2] + \s[X_1 | X_1+Y_2 ; Y_1 | Y_1 + X_2] \\ 
&=  2 \s[X : Y] \geq 2(\eta_0 - 2\varepsilon_0)(\H[X]+\H[Y]) \\ 
&= (\eta_0 - 2\varepsilon_0)(\H[X_1 | X_1+Y_2] + \H[Y_1 | Y_1+X_2]) + \\
&\qquad (\eta_0-2\varepsilon_0)(\H[X_1+Y_2] + \H[Y_1+X_2]).
\end{align*}
Use the upper bound on $\s[X_1+Y_2 ; Y_1+X_2]$ from the endgame inequalities to estimate
\begin{align*}
\E_{u,w} \s[X_u  ; Y_w] &\geq (\eta_0 - 2\varepsilon_0)(\H[X_1 | X_1+Y_2] + \H[Y_1 | Y_1+X_2]) - 12\varepsilon_0 (\H[X] + \H[Y]) \\ 
&\stackrel{\eqref{eq:DoublingMassFiberSize}}{\geq} (\eta_0 - 2\varepsilon_0) 2\s[X ; Y] - 12\varepsilon_0 (\H[X] + \H[Y]) \\ 
&\geq (2(\eta_0 - 2\varepsilon_0)^2 - 12\varepsilon_0) (\H[X] + \H[Y]).
\end{align*}
Enforce 
\[
\varepsilon_0 \leq \frac{\eta_0}{4}\quad\text{and}\quad \varepsilon_0 \leq \frac{\eta_0^2}{48}
\]
so that 
\[
\E_{u,w} \s[X_u  ; Y_w] \geq \frac{\eta_0^2}{4} (\H[X] + \H[Y]).
\]
Next, enforce 
\[
2^{13}\varepsilon_0 \leq  \frac{\eta_0^2}{8} \Longleftrightarrow \varepsilon_0 \leq 2^{-15} \eta_0^2
\]
so that 
\[
\E_{u,w} \s[X_u | \pi_{V(u,w)}(X_u) ; Y_w | \pi_{V(u,w)}(Y_w)] \geq \frac{\eta_0^2}{8} (\H[X] + \H[Y]). 
\]
By~\Cref{lem:LocalToGlobal}, there exists a subspace $\bar V$ with 
\[
\H[\U_{\bar V}] \leq 8\cdot 64\cdot \eta_0^{-4} \E_{u,w} 7(\H[X_u] + \H[Y_w]) \leq 2^{12} \eta_0^{-4}  (\H[X] + \H[Y])
\]
and 
\[
\H[\pi_{\bar V}(X)] + \H[\pi_{\bar V}(Y)] \leq (1 - \frac{\eta_0^2}{32}) (\H[X] + \H[Y]) \leq (1 - \frac{\eta_0^2}{32}) (\H[\tilde X] + \H[\tilde Y]).
\]
This satisfies the conclusions with
\begin{equation}\label{eq:Case2Bounds}
\begin{split}
    L_1 &= 2^{12} \eta_0^{-4} + \frac{4}{\varepsilon_0} L_0, \\ 
    c &= \frac{\eta_0^2}{32}.
\end{split}
\end{equation}
Considering the Case 1 and Case 2 bounds~\eqref{eq:Case1Bounds}, and the Case 3 bounds~\eqref{eq:Case2Bounds}, gives~\eqref{eq:OverallBounds}. 
\end{proof}

\begin{theorem}\label{thm:statement-B}
There exists a constants $C > 0$ such that $\mb B(\eta,  \eta^2, \eta^{- C \eta^{-1}})$ holds for all $\eta \in (0,1/2]$.
\end{theorem}
\begin{proof}
The base case for the induction is  
\begin{align*}
    \mb B(\frac{1}{2}, 0, 0) \text{ holds}. 
\end{align*}
Indeed, for any independent random variables $X,Y$, 
\[
\H[X+Y] \geq \max\{\H[X], \H[Y]\} \geq \frac{1}{2}(\H[X] + \H[Y]). 
\]
Let $\eta_0 \in (0, 1/2]$ and $\eta_1 = \eta_0 - C_0^{-1} \eta_0^2$, where $C_0 = 2^{15}$ is the constant from~\Cref{prop:InductiveStep}.  
\Cref{prop:InductiveStep} and \Cref{lem:BasicRelationsAB} together give
\begin{align*}
    \mb B(\eta_0, C_0^{-1} \eta_0^2, L_0) &\Longrightarrow \mb A(\eta_1, C_0^3 \eta_0^{-4} \max\{L_0, 1\}, C_0^{-1} \eta_0^2) \\ 
    &\Longrightarrow \mb B\bigl(\eta_1 ,\; C_0^{-1} \eta_1^2,\; \lceil \frac{ \log C_0 \eta_1^{-2} }{C_0^{-1} \eta_0^2} \rceil C_0^3 \eta_0^{-4} \max\{L_0, 1\}\bigr). 
\end{align*}
Let $L(\eta)$ be the best constant so that $\mb B(\eta, C_0^{-1} \eta^2, L(\eta))$ holds. The inductive step implies that, for $\eta + C_0^{-1} \eta^2 \leq 1/2$,
\begin{align*}
L(\eta) &\leq 2\, C_0^5 \eta^{-6} \max\{L(\eta + C_0^{-1}\eta^2), 1\}.
\end{align*}
Iterating this inequality implies that, for $2\eta \leq 1/2$, 
\[
L(\eta) \leq (2C_0^5 \eta^{-6})^{2C_0 \eta^{-1}} \max\{ L(2\eta), 1\} \leq \eta^{-C_1 \eta^{-1}} \max\{L(2\eta), 1\}
\]
where $C_1 > 0$ is a new constant. 
Iterating this inequality logarithmically many times yields 
\[
L(\eta) \leq \eta^{-C_1 \eta^{-1} - \frac{1}{2} C_1 \eta^{-1} - \frac{1}{4} C_1 \eta^{-1} - \cdots} \leq \eta^{-2C_1 \eta^{-1}}
\]
as desired. 
\end{proof}

Theorem \ref{thm:main-entropy} now follows from Theorem \ref{thm:statement-B} by setting $\eta = \varepsilon/2$ and $L = \eta^{-C \eta^{-1}}$.%\footnote{This upper bound on $L$ is nearly optimal: by analysing the Bernoulli example one can show that the function $L$ in Theorem \ref{thm:main-entropy} must be at least exponential in $\varepsilon^{-1}$.}

\begin{proof}[Proof of Corollary \ref{cor:rich-cosets}]
    Let $X, Y$ be independent and let $s = \s[X; Y] = \H[X]+\H[Y]-\H[X+Y]$. By Theorem \ref{thm:main-entropy} there is a subspace $V \subset G$ with $\H[\U_V] \le L(\H[X]+\H[Y])$ and $\s[\pi_V(X); \pi_V(Y)] \le \varepsilon(\H[X]+\H[Y])$.
    So by (\ref{eq:TrivialEstimateDoublingMass}) and (\ref{eq:FibringIdentity}) we get 
    \[
     \min(\H[X|\pi_V(X)], \H[Y|\pi_V(Y)]) \ge \s[X|\pi_V(X) ; Y|\pi_V(Y)] \ge s - \varepsilon(\H[X]+\H[Y]).
    \]
\end{proof}

\begin{proof}[Proof of Corollary \ref{cor:many-sums}]
Fix $\varepsilon>0$ and let $X_1, \ldots, X_k$ be independent random variables. Let $\delta= \frac{\varepsilon}{k-1}$. 
We iteratively apply Corollary \ref{cor:rich-cosets} as follows. Start with $W_0=\{0\}$ and suppose that at step $i\ge 0$ we have constructed some subspace $W_i\subset G$.
If for some $j \in\{1, \ldots, k-1\}$ we have 
\begin{align*}
\H[\pi_{W_i}(X_1)+\ldots+\pi_{W_i}(X_{j}) +\pi_{W_i}(X_{j+1})] \le \H[\pi_{W_i}(X_1)+\ldots+\pi_{W_i}(X_{j})]+ \H[\pi_{W_i}(X_{j+1})] \\- \delta (\H[X_1]+\ldots+\H[X_k])    
\end{align*}
then we apply Corollary \ref{cor:rich-cosets} with $\varepsilon = \delta/2$, $X = X_1+\ldots+X_j$ and $Y=X_{j+1}$ and find a subspace $V_{i+1} \subset G/W_{i}$ so that for some $L = L(\delta/2)$ we have $\H[\U_{V_{i+1}} ] \le L (\H[X_1]+\ldots+\H[X_k])$ and 
for $W_{i+1} = \pi_{W_i}^{-1}(V_{i+1})$ we have
\[
\H[\pi_{W_{i+1}}(X_{j+1})] \le \H[\pi_{W_{i}}(X_{j+1})] - \frac{\delta}2 (\H[X_1]+\ldots+\H[X_k]).
\]
Since the total entropy of $(\pi_{W_i}(X_1), \ldots, \pi_{W_i}(X_k))$ drops by at least $\frac{\delta}2 (\H[X_1]+\ldots+\H[X_k])$ at every step, this process must stop in $m\le \frac{2}{\delta}$ steps. For the resulting subspace $V=W_m$ we then obtain:
\[
\H[\U_V] \le m L (\H[X_1]+\ldots+\H[X_k]) \le \frac{2}{\delta}L (\H[X_1]+\ldots+\H[X_k])
\]
and
\begin{align*}
\H[\pi_V(X_1+\ldots+X_k)] \ge \H[\pi_V(X_1+\ldots+X_{k-1})] + \H[\pi_V(X_k)] - \delta (\H[X_1]+\ldots+\H[X_k])    \\
\ge \ldots \ge \H[\pi_V(X_1)]+\ldots + \H[\pi_V(X_k)] - (k-1) \delta (\H[X_1]+\ldots+\H[X_k])
\end{align*}
and by $\delta = \frac{\varepsilon}{k-1}$ this gives us the desired bound. 
\end{proof}

\end{document}